\documentclass[12pt,reqno,a4paper]{amsart}
\usepackage{amsmath,amssymb,amsfonts,amscd}
\usepackage[mathscr]{eucal}
\usepackage{comment}
\usepackage{color}
\usepackage{hyperref}

\dedicatory{To the memory of Kirill Mackenzie (1951--2020)}
\date{}
\thanks{This work was supported by the National Science Foundation under
grant No.1708033.}   

\title[BV operator generating higher Koszul brackets on  forms]{On  a Batalin--Vilkovisky  operator generating higher Koszul brackets on differential forms}
\author{Ekaterina~Shemyakova}
\address{Department of Mathematics,  University of Toledo, Toledo,  Ohio, USA}
\email{ekaterina.shemyakova@utoledo.edu}

\newtheorem{theorem}{Theorem}[section]

\newtheorem{lemma}{Lemma}[section]

\theoremstyle{definition}
\newtheorem{definition}{Definition}[section]
\newtheorem{example}{Example}[section]
\newtheorem{remark}{Remark}[section]

\def\co{\colon\thinspace}

\renewcommand{\leq}{\leqslant}
\renewcommand{\geq}{\geqslant}

\DeclareMathOperator{\DO}{DO_{\hbar}}
\DeclareMathOperator{\fDO}{\widehat{\DO}}
\newcommand{\hi}{\frac{\hbar}{i}}
\newcommand{\lhi}{{\hbar}/{i}}
\newcommand{\ih}{\frac{i}{\hbar}}
\newcommand{\hp}{{\hat p}}
\DeclareMathOperator{\symb}{\sigma}
\DeclareMathOperator{\Dens}{\mathfrak{Dens}}

\newcommand{\dr}{\delta_{\rho}}

\newcommand{\Dbar}{%
   {{D\mkern-14 mu
   \mathchoice{\raisebox{-2pt}{$\displaystyle \mathchar'26$}}
             {\raisebox{-2pt}{$\mathchar'26$}}
             {\raisebox{-1pt}{$\scriptstyle \mathchar'26$}}
             {\raisebox{-0.5pt}{$\scriptscriptstyle \mathchar'26$}}}
             \mkern 5 mu}}

\DeclareMathOperator{\fun}{\mathit{C^{\infty}}}

\DeclareMathOperator{\Ber}{Ber}
\DeclareMathOperator{\Vol}{Vol}

\DeclareMathOperator{\sign}{sgn}

\DeclareMathOperator{\ad}{ad}

\newcommand{\der}[2]{{\frac{\partial {#1}}{\partial {#2}}}}

\newcommand{\lder}[2]{{\partial {#1}/\partial {#2}}}

\newcommand{\RR}{\mathbb R}
\newcommand{\CC}{\mathbb C}

\newcommand{\ZZ}{{\mathbb Z}}

\newcommand{\p}{\partial}

\newcommand{\w}{{\mathbf{w}}}

\renewcommand{\a}{\alpha}
\renewcommand{\b}{\beta}
\newcommand{\e}{\varepsilon}
\newcommand{\s}{\sigma}

\renewcommand{\O}{\Omega}
\newcommand{\D}{\Delta}
\renewcommand{\o}{\omega}

\newcommand{\la}{{\lambda}}

\renewcommand{\d}{\delta}

\newcommand{\ft}{{\tilde f}}

\newcommand{\itt}{{\tilde \imath}}

\newcommand{\at}{{\tilde a}}

\newcommand{\ct}{{\tilde c}}

\newcommand{\xt}{{\tilde x}}

\newcommand{\Lt}{{\tilde L}}

\newcommand{\omt}{{\tilde \omega}}
\newcommand{\sit}{{\tilde \sigma}}

\DeclareMathOperator{\Vect}{\mathrm{Vect}}

\DeclareMathOperator{\Mult}{\mathfrak{A}}

\newcommand{\lsch}{{[\![}}
\newcommand{\rsch}{{]\!]}}

\newcommand{\Sinf}{S_{\infty}}
\newcommand{\Pinf}{P_{\infty}}
\newcommand{\Linf}{L_{\infty}}

\newcommand{\void}{\varnothing}

\newcommand{\tto}{{\linethickness{2pt}
		  \,\begin{picture}(1,0)
                   \put(0,0.26){\line(1,0){0.95}}
                   \put(0,0){$\boldsymbol{\rightarrow}$}
                  \end{picture}
                  }\,
}

\begin{document}
\begin{abstract} We introduce a formal $\hbar$-differential operator $\Delta$ that generates higher Koszul brackets on the algebra of (pseudo)differential forms on a $P_{\infty}$-manifold. Such an operator was first mentioned by Khudaverdian and Voronov in \texttt{arXiv:1808.10049}. (This operator is an analogue of the Koszul--Brylinski boundary operator $\partial_P$ which defines Poisson homology for an ordinary Poisson structure.)

Here we introduce $\Delta=\Delta_P$ by a different method and establish its properties. We show that this BV type operator generating higher Koszul brackets can be included in a one-parameter family of BV type formal $\hbar$-differential operators, which can be understood as a quantization of the cotangent $L_{\infty}$-bialgebroid. We obtain symmetric description
on both $\Pi TM$ and $\Pi T^*M$.

For the purpose of the above, we develop in detail a theory of formal $\hbar$-differential operators and also of operators acting on densities on dual vector bundles. In particular, we have a statement about operators that can be seen as a quantization
of the Mackenzie--Xu canonical diffeomorphism. Another interesting feature is that we are able to introduce a grading, not a filtration, on our algebras of operators. When operators  act on objects on vector bundles, we obtain a bi-grading.

\end{abstract}

\maketitle

\section{Introduction}\label{sec.intro}
\emph{Higher Koszul brackets} generalize the classical (binary) Koszul bracket on forms on a Poisson manifold. They are defined on the algebra of pseudodifferential forms on a supermanifold with a $\Pinf$- or homotopy Poisson structure. Their construction was  introduced in 2008 by H.~ Khudaverdian and Th.~Voronov~\cite{tv:higherpoisson}.

It has been noticed by Th.~Voronov~\cite{tv:graded,tv:qman-mack,tv:highkosz} that a Lie algebroid structure and similar structures like $\Linf$-algebroids
can be seen as ideal objects which ``manifest'' themselves in different equivalent ways on ``neighbor'' vector bundles obtained from each other by dualization and parity reversion functor $\Pi$.

From this viewpoint, the Koszul bracket
is a manifestation on $\Pi TM$
of the Lie algebroid structure on the cotangent bundle $T^*M$ induced by a Poisson structure on $M$ and is the corresponding Lie--Schouten bracket, see
K.~Mackenzie~\cite[Ch.10]{mackenzie:book2005}.
An equivalent manifestation on $\Pi T^*M$ of the same Lie algebroid structure is the Poisson-Lichnerowicz differential $d_P\co \Mult^{k}(M)\to \Mult^{k+1}(M)$. (Here $\Mult^{k}(M)$ denotes multivector fields  of degree $k$ on a manifold $M$.)

In the similar way, higher Koszul brackets are one of the equivalent manifestations of the $\Linf$-algebroid structure on the cotangent bundle of a $\Pinf$-manifold.

 It was shown by J.-L.~Koszul~\cite{koszul:crochet85} that classical Koszul bracket
can be generated by a second order ``Batalin--Vilkovisky (BV) type''
differential operator $\p_P$, known as the  Koszul--Brylinski differential (it defines Poisson homology dual to Poisson cohomology~\cite{brylinski1988}). In the appendix to the recent paper~\cite{tv:highkosz}, it was briefly indicated how Koszul's construction can be extended to the ``higher'' case.

In the present letter, we introduce a BV type operator generating higher Koszul brackets by a different method and establish its properties. It is a formal $\hbar$-differential operator of (in general) infinite order. We develop a theory of such operators. Using this operator, we can consider  ``quantum higher Koszul brackets'' depending on parameter $\hbar$. 

Further, recall that in the classical situation of a Poisson manifold, the cotangent Lie algebroid  is actually a Lie bialgebroid (if one takes into account the tautological Lie algebroid structure of $TM$). This also generalizes to the higher case as an $\Linf$-bialgebroid~\cite{tv:nonlinearpullback},\cite{tv:linfbialg}. Here we show that the BV   operator generating higher Koszul brackets can be included in a one-parameter family of formal $\hbar$-differential operators which we can interpret as a structure of a ``quantum $\Linf$-bialgebroid''.

To keep the symmetry between the manifestations existing on the classical level, on the quantum level we need to introduce into consideration
densities of different weights.

The structure of the letter is as follows. In Section~\ref{sec.origin}, we recall the classical Koszul bracket for a Poisson manifold and the BV operator generating it. Then we recall the construction of the sequence of ``higher  Koszul brackets'' for a homotopy Poisson or $\Pinf$-structure on a supermanifold due to Khudaverdian--Voronov\,\footnote{``Higher brackets'' means that besides a binary bracket, there a $3$-bracket, a $4$-bracket, etc.; but it also includes  ``lower'' order brackets such as   $0$- and   $1$-brackets.}.

In Section~\ref{sec.operator}, we define (following~\cite{tv:highkosz}) an operator serving as an analogue of Koszul's second order BV type operator, for higher Koszul brackets. It looks superficially similar to the classical operator. However, it is, in general, a differential operator of infinite order. More precisely,  it is a \emph{formal $\hbar$-differential operator}. We need  the theory of such operators,  which is  interesting also for broader purposes.

We develop such a theory in  subsection~\ref{subsec.formalop}. In particular, we show that any such an operator (of an infinite order from the conventional viewpoint) has a principal symbol, which is  a well-defined  formal function  on $T^*M$. Planck's constant $\hbar$ is treated as a formal variable and one of the generators of our algebra. This makes it possible to introduce a grading, rather than a filtration, in the space of our operators. For   a formal $\hbar$-differential operator,  we consider the sequences of ``quantum'' and ``classical'' brackets (first introduced in~\cite{tv:higherder}) and prove that the master Hamiltonian for the sequence of classical brackets is exactly the principal symbol.

Using the results above, we construct formal $\hbar$-differential operators that can be seen as  quantizations  of the classical Hamiltonians arising in the definition of higher Koszul brackets.

This immediately gives the following main statement: there exists an odd formal $\hbar$-differential operator $\D_P$ acting on pseudodifferential forms and generating higher Koszul brackets and which also   obeys $\D_P^2=0$ (``a BV operator for higher Koszul brackets''). Moreover, we show that the constructed BV operator can be combined with the operator $-i\hbar\,d$ (where $d$ is the de Rham differential\footnote{\,The purpose of inserting the factor $-i\hbar$ is to have an $\hbar$-differential operator.}) into a one-parameter family of BV operators. We interpret it as a ``quantization'' of the structure of the cotangent $\Linf$-bialgebroid of a $\Pinf$-manifold, in the manifestation on the bundle $\Pi TM$. We then look into a similar quantization in the dual manifestation on the bundle $\Pi T^*M$. It turns out that for a natural construction of a BV operator on $\Pi T^*M$ there is an obstruction, which is the modular class of a given $\Pinf$-structure. The formula can be remedied by a correction term, but we lose the symmetry between the manifestations on $\Pi TM$ and $\Pi T^*M$ that exist  classically.

In Section~\ref{sec.further}, we are concerned with developing the desired symmetric picture. For that, we have to depart from   operators acting on functions on a supermanifold and consider the more general case of operators acting on \emph{densities} of various weights. As the goal is to obtain BV operators for $\Linf$-bialgebroids in dual manifestations (note that we do not stick to one formal definition of what an $\Linf$-bialgebroid is, since it looks to be still a rather open question; but of course there is some general idea), we need to develop a theory of operators for dual bundles.

We use the fiberwise Fourier transform discovered by Th.~Voronov and A.~Zorich~\cite{tv:ivb} for the purposes of supermanifold integration theory. We use an $\hbar$-version of it for fiberwise densities.

And we also   introduce a new $\ZZ\times \ZZ$-grading for $\hbar$-differential operators on a vector bundle. (This $\ZZ\times \ZZ$-grading of operators can be seen as a quantum analogue of the graded manifold approach to the double vector bundle structure of the cotangent of a vector bundle as in~\cite{tv:qman-mack,tv:qman-esi}.) We prove that   the fiberwise $\hbar$-Fourier transform induces anti-isomorphisms of algebras of formal  $\hbar$-differential operators on dual vector bundles acting on densities of suitable weights. We also show   that the $\ZZ\times \ZZ$-grading undergoes a mirror reflection. In particular, this gives a quantum version of the fundamental Mackenzie--Xu canonical antisymplectomorphism~\cite{mackenzie:bialg}. (Which is induced in the limit $\hbar\to 0$.) We also show that the fiberwise $\hbar$-Fourier transform preserve the class of ``quantum pullbacks'' of Th.~Voronov~\cite{tv:microformal}, which are particular type of Fourier integral operators, with an explicit description of the action on  phase functions. In subsection~\ref{subsec.applcot}, we return from general theory to the concrete situation of our interest. We again consider the cotangent $\Linf$-bialgebroid for a $\Pinf$-manifold. The difference with the approach in Section~\ref{sec.operator}  is that now we construct BV type operators acting on half-densities, not on functions. (Worth mentioning that in the actual physical BV setup, the canonical odd Laplacian due to Khudaverdian~\cite{hov:semi} acts on half-densities.) This way we achieve a complete symmetry in the manifestations on $\Pi T^*M$ and $\Pi TM$. The modular class of a $\Pinf$-structure that was popping up in another approach, here has a different role; it seems to be an obstruction to a ``quantum lift'' of the anchor (such a lift would be given by an operator of ``quantum pullback'' type mentioned above).

\emph{Notation.}
Here ``forms'' and ``multivector fields'' on ordinary or super manifold, are understood as functions on the supermanifolds $\Pi TM$ and $\Pi T^*M$ respectively. Hence our ``forms'' are actually
pseudodifferential forms in the supercase and inhomogeneous differential forms in the ordinary case.
Similarly for multivector fields.

By $\O(M)$ we denote the algebra of forms,
by $\Mult(M)$ the algebra of multivector fields, and by $\Mult^k(M)$ multivectors of degree $k$ or $k$-vectors.
All super formulas are written for quantities that are homogeneous in the sense of parity, which is denoted by the tilde over a symbol, e.g. $\xt=0,1$ for $x$ even or odd. If $x^a$ are local coordinates on a (super)manifold, then forms and multivector fields are functions of the variables $x^a, dx^a$ and $x^a,x^*_a$ respectively. The variables $dx^a$ and $x^*_a$ have parities opposite to those of the corresponding coordinates: $\widetilde{dx^a}=\widetilde{x^*_a}=\at+1$, where $\at=\widetilde{x^a}$. Under a change of coordinates, the variables $x^*_a$ transforms in the same way as the partial derivatives $\p_a$. The canonical Poisson bracket (even) on $T^*M$ for a supermanifold $M$ is denoted $\{-,-\}$, while the canonical Schouten bracket on $\Pi T^*M$ (odd) is denoted $\lsch -, -\rsch$.


\emph{Acknowledgement.} The author is very thankful to Theodore Voronov for attracting her attention to this problem, as well as for valuable comments and discussions.  Many thanks also go to the referee for comments that helped to improve the exposition. 

\section{Recollection of higher Koszul brackets}\label{sec.origin}

\subsection{Binary Koszul bracket.}
The classical (binary) Koszul bracket on differential forms on a Poisson manifold was introduced by J.-L.~Koszul~\cite{koszul:crochet85}; the case of $1$-forms was known earlier, in particular, in the context of integrable systems. (See more in~\cite{yvette:modular2008}.) Below
are the main formulas (which we give in a version generalized to supermanifolds, see~\cite{tv:higherpoisson}). Let $\{f,g\}_P$ denote the Poisson bracket corresponding to a Poisson bivector $P=\frac{1}{2}\,P^{ab}(x)x^*_bx^*_a$, then the \emph{Koszul bracket} on forms denoted $[\o,\s]_P$ is given by
\begin{equation}\label{eq.binkosz}
    [f,g]_P:=0\,, \quad [df,g]_P:=\{f,g\}_P\,, \quad [df,dg]_P:=(-1)^{\ft}d\{f,g\}_P
\end{equation}
together with the symmetry, linearity and Leibniz conditions:
\begin{align}\label{eq.binkoszlin}
    [c\o,\s]_P&=(-1)^{\ct}c[\o,\s]_P\,,\\
    [\o,\s]_P&=(-1)^{\omt\sit}[\s,\o]_P\,, \label{eq.binkoszsym}\\
    [\o,\s\tau]_P&=[\o,\s]_P\tau+(-1)^{(\omt+1)\sit} \s[\o,\tau]_P\,, \label{eq.binkoszlei}
\intertext{which imply the Jacobi identity in the form}
    [\o,[\s,\tau]_P]_P&=(-1)^{\omt+1} [[\o,\s]_P,\tau]_P+(-1)^{(\omt+1)(\sit+1)} [\s,[\o,\tau]_P]_P\,. \label{eq.binkoszjac}
\end{align}
(The conditions~\eqref{eq.binkoszlin}, \eqref{eq.binkoszsym}, \eqref{eq.binkoszlei}, \eqref{eq.binkoszjac} together   mean  that $\O(M)$ is a ``Schouten'' or ``odd Poisson'' algebra under the Koszul bracket. Y.~Kosmann-Schwarzbach~\cite{yvette:derived-survey2004} refers to the Koszul bracket   as  ``Koszul-Schouten bracket''.) Note that the sign conventions followed here may be different from other sources; we chose those convenient for generalization to   higher  brackets.   
The Koszul bracket also satisfies
\begin{equation}\label{eq.binkoszder}
    d[\o,\s]_P = -[d\o,\s]_P+(-1)^{\tilde\o+1}[\o,d\s]_P\,.
\end{equation}

The derivation property~\eqref{eq.binkoszder} expresses the fact that $T^*M$ is actually a Lie bialgebroid (see~\cite{mackenzie:book2005} and references therein).

Koszul showed in~\cite{koszul:crochet85} that the bracket $[\o,\s]_P$ can be obtained from an odd second-order differential operator $\D$,
\begin{equation} \label{eq.koszulbryl}
    \D=\partial_P:=[d, i(P)]
\end{equation}
on the algebra $\O(M)$ by the formula
\begin{equation}\label{eq.binkoszbv}
    \D(\o\s)=\D(\o)\s +(-1)^{\tilde\o}\o\D(\s) + [\o,\s]_P
\end{equation}
Here $i(P)$ is the operator of the interior product with the bivector $P$. It decreases the degrees of forms by two, hence its commutator with $d$ is of degree $-1$. The operator $\partial_P$ is often referred to as the ``Koszul--Brylinski differential''. It plays the role of the boundary operator in the definition of Poisson homology~\cite{brylinski1988},  a dual notion    to   Poisson cohomology defined by the Lichnerowicz differential $d_P$ (similarly with the relation between Lie algebra homology and Lie algebra cohomology).

Equation~\eqref{eq.binkoszbv} is similar in form to the  analogous equation relating the canonical Schouten bracket on the algebra of multivector fields $\Mult(M)$ and the divergence operator $\d=\dr$ on multivector fields (defined using   a choice of a volume element $\rho$),
\begin{equation}\label{eq.schoutbv}
    \d(TS)=\d (T) S +(-1)^{\tilde T} T\d (S) +\lsch T,S\rsch\,,
\end{equation}
which is classically known (see e.g.~\cite{kirillov:invariant}).

Although we are used to thinking about divergence as a first order operator,
from the viewpoint of supermanifolds  $\d$ is an odd differential operator of order \emph{two},
\begin{equation}
    \d T=(-1)^{\at}\frac{1}{\rho(x)}\der{}{x^a}\left(\rho(x)\der{T}{x^*_a}\right)
\end{equation}
for a multivector field $T=T(x,x^*)$ (see e.g.~\cite[Ch.5]{tv:git}).
Since their appearance in quantum field theory~\cite{bv:perv,bv:vtor}, odd differential operators  generating an odd bracket by a formula like~\eqref{eq.binkoszbv} or \eqref{eq.schoutbv} are referred to as ``Batalin--Vilkovisky type operators''. (Geometric meaning of these operators
as odd analogues of the ordinary Laplacians
was first understood by H.~Khudaverdian~\cite{hov:deltabest} and A.S.~Schwarz~\cite{schwarz:bv}.
See also \cite{hov:khn1}, \cite{yvette:divergence}, \cite{tv:laplace1}, \cite{tv:laplace2}.)

Y.~Kosmann-Schwarzbach~\cite{yvette:divergence}
showed that the Koszul operator $\D$ is a BV operator in the narrow sense, i.e. it is the odd Laplacian constructed from a bracket and a volume element, where the bracket is the Koszul bracket and the volume element is the canonical invariant volume element on $\Pi TM$.

In the next subsection we recall a ``higher analogue'' of the classical Koszul bracket (when a Poisson structure is replaced by a $\Pinf$-structure) and in Section~\ref{sec.operator} we shall introduce and study an analogue for a higher setting of Koszul's BV operator. Unlike the second-order Batalin--Vilkovisky operators related with binary brackets, this will be a formal differential operator of infinite order. Binary Koszul
bracket can be (and in fact originally was)
defined on an ordinary manifold. As for higher Koszul brackets, for the theory to be nontrivial one really needs supermanifolds.

\subsection{Higher Koszul brackets for a $\Pinf$-structure.} \label{subsec.highkosz}
Let $M$ be  a supermanifold. A homotopy Poisson structure or $\Pinf$-structure is given on $M$ by an even function on $\Pi T^*M$, denote it $P$, which satisfies the equation $\lsch P,P\rsch=0$. (We only need   its infinite jet at $M\subset \Pi T^*M$, i.e. the power expansion in the variables $x^*_a$.) It generates the sequence of higher Poisson brackets on functions on $M$ by
Th.~Voronov's higher derived bracket formulas~\cite{tv:higherder},
\begin{equation}\label{eq.highpoiss}
    \{f_1,\ldots,f_n\}_P:= \lsch \ldots \lsch P,f_1\rsch, \ldots, f_n\rsch_{|M}
\end{equation}
(the restriction on $M$ at the right-hand side simply means setting $x^*_a=0$). Here $n=0,1,2,3\ldots$ The brackets~\eqref{eq.highpoiss} have alternating parities, in particular, the binary bracket is even. Condition $\lsch P,P\rsch=0$ is equivalent to the sequence of ``higher Jacobi identities'' for the brackets~\eqref{eq.highpoiss}.

Notice that $\Pinf$-structures are not some abstract generalization. They can naturally arise
in ordinary Poisson geometry. A.~Cattaneo and G.~Felder~\cite{cattaneo-felder:relative2007} discovered $\Pinf$-structures
from deformations of coisotropic submanifolds of an ordinary Poisson manifold.

Analogue of the Koszul bracket
for a $\Pinf$-structure was found by H.~Khudaverdian and Th.~Voronov~\cite{tv:higherpoisson}. It is a sequence of odd symmetric brackets on the algebra $\O(M)$, notation $[\o_1,\ldots,\o_n]_P$, where $\o_i\in \O(M)$ and $n=0,1,2,3,\ldots$ , defined on functions and exact $1$-forms by the equations
\begin{equation}\label{eq.highkoszini}
    [df_1, \ldots, df_{n-1}, f_n]_P:=\{f_1,\ldots,f_n\}_P\,,\quad [df_1,df_2,\ldots,df_n]_P:=(-1)^{\ft_1}d\{f_1,f_2,\ldots,f_n\}_P \ ,
\end{equation}
and for all other combinations the brackets are zero, and then extended to the whole algebra $\O(M)$ by the Leibniz rule as multiderivations. The resulting brackets are called \emph{higher Koszul brackets}.
They also showed
that the higher Jacobi identities for Poisson brackets giving $\Pinf$-structure
imply the higher Jacobi identities
for the higher Koszul brackets, so  they form what is known as an $\Sinf$-structure on the algebra $\O(M)$.  Note that a $\Pinf$-structure is a generalization of an even Poisson structure, while an $\Sinf$-structure is a generalization of an odd Poisson structure. In particular, brackets in the $\Pinf$ case are antisymmetric and of alternating parities while in the $\Sinf$ case they are symmetric and all odd.

In other words, the higher Koszul brackets for a $\Pinf$-structure on $M$ is an $\Sinf$-structure for the supermanifold $\Pi TM$. They are the Lie--Schouten brackets corresponding to an $\Linf$-algebroid structure in the cotangent bundle $T^*M$ induced by a $\Pinf$-structure on $M$.

Any $\Sinf$-structure on a supermanifold is given~\cite{tv:higherder} by an odd function $H$ on its cotangent bundle   satisfying $\{H,H\}=0$ (an odd \emph{master Hamiltonian}).
This master Hamiltonian $H_P \in \fun(T^*(\Pi TM))$ for higher Koszul brackets  is constructed as follows~\cite{tv:higherpoisson,tv:highkosz}. We take
the Hamiltonian lift of  the homological vector on $\Pi T^*M$ and apply to it
the natural antisymplectomorphism $T^*(\Pi TM)\cong T^*(\Pi T^*M)$. This vector field $d_P\in \Vect(\Pi T^*M)$, which  is the analogue of the Lichnerowicz differential, is given by $d_P:=\lsch P,-\rsch$. 
One may now note that the canonical Schouten bracket on $\Pi T^*M$ is itself derived from a master Hamiltonian, which  is the invariant quadratic function $D^*\in T^*(\Pi T^*M)$,
\begin{equation}\label{eq.masterhamforsch}
    D^*=(-1)^{\at}\pi^ap_a\,,
\end{equation}
where $p_a$, $\pi^a$ are the conjugate momenta for the variables $x^a,x^*_a$.  The function $D^*$ on $T^*(\Pi T^*M)$ corresponds under the identification $T^*(\Pi TM)\cong T^*(\Pi T^*M)$ to the function $D =dx^ap_a$ on $T^*(\Pi TM)$, the lift the de Rham differential. Therefore, vector field $d_P\in \Vect(\Pi T^*M)$ lifts to the Hamiltonian
\begin{equation}\label{eq.hamlichn}
    H_P^*=\{D^*, P\}\,.
\end{equation}
By using the   identification $T^*(\Pi TM)\cong T^*(\Pi T^*M)$, which preserves the canonical Poisson bracket on the cotangent bundle up to a sign, we arrive at
\begin{equation}\label{eq.masterkosz}
    H_P =\{D, P^*\}\,.
\end{equation}
and this is the sought-for master Hamiltonian for the higher Koszul brackets. Here $P^*\in \fun(T^*(\Pi TM))$ is obtained from $P=P(x,x^*)$ by substituting $x^*_a=\pi_a$. It is also possible to write down an explicit formula for $H_P$,
\begin{equation}
    H_P=dx^a \der{P}{x^a}(x,\pi) + (-1)^{\at} \der{P}{x^*_a}(x,\pi)p_a\,
\end{equation}
(see~\cite{tv:higherpoisson}), but we do not need it.

 \begin{remark} We are using many times
 the natural diffeomorphism $T^*E\cong T^*(E^*)$ for an arbitrary vector bundle $E$ preserving the canonical Poisson brackets up to a sign, which was discovered in~\cite{mackenzie:bialg} and is called the \emph{Mackenzie--Xu transformation}. See its generalization to the super case and an odd version in~\cite{tv:graded,tv:qman-mack}.
 \end{remark}

\section{Generating BV operator}\label{sec.operator}
\subsection{Construction of the operator}
Let $M$ be a supermanifold endowed with a $\Pinf$-structure. Our goal is to prove that the following formula gives an operator  generating the higher Koszul brackets on $\O(M)$\,:
\begin{equation}\label{eq.hdelp}
    \D_P:=[d, \hat P]\,.
\end{equation}
Here $P$ is the Poisson tensor specifying a $\Pinf$-structure on $M$. Recall that it is an even function on $\Pi T^*M$ satisfying $\lsch P,P\rsch=0$.  Notation $\hat P$ has the following meaning. For an arbitrary multivector field $P=P(x,x^*)$,
\begin{equation}\label{eq.Phat}
    \hat P:=P\left(x,\hi\der{}{dx}\right)\,.
\end{equation}
It is a ``vertical'' formal $\hbar$-differential operator (see the next subsection) on $\O(M)$ canonically corresponding to $P$.
(It is very close to the standard notion of the interior product $i(P)$ defined for    multivectors of   fixed degree,  from which it differs   by the factor of $(\lhi)^k$, where $k$ is the degree.) A quick way of defining $\hat P$ is by the Berezin integral
\begin{equation}\label{eq.Phatint}
    (\hat P\o)(x,dx)=\int \Dbar(x^*)D(dx')e^{\ih(dx-dx')x^*}P(x,x^*)\,\o(x,dx')\,
\end{equation}
(see more in subsection~\ref{subsec.dualiz}). Integration in~\eqref{eq.Phatint} over $\Pi T_x^*M\times \Pi T_xM$.
Here   notation such as $\Dbar(x^*)$ means the coordinate volume element normalized by  the factor arising in the inverse Fourier transform.

The exact statement is as follows.

\begin{theorem} \label{thm.DP}
The operator $\D_P$ defined by~\eqref{eq.hdelp} is a formal $\hbar$-differential operator on the algebra $\O(M)$. The sequence of classical brackets generated by $\D_P$ is   the higher Koszul brackets corresponding to a $\Pinf$-structure given by $P$.
\end{theorem}

In other words, we claim that
\begin{equation}\label{eq.highkoszbydp}
    [\o_1,\ldots,\o_n]_P=\lim\limits_{\hbar\to 0}(-i\hbar)^{-n}[\ldots[\D_P,\o_1],\ldots,\o_n](1)
\end{equation}
for all $n=0,1,2,\ldots$  and $\o_i\in\O(M)$\,. At the right-hand side we identify a differential form $\o$ with the  operator of multiplication by $\o$ on the algebra $\O(M)$, and the result is evaluated at $1\in \O(M)$.

\begin{example}\label{ex.ordpoiss}
Suppose our $\Pinf$-structure is an ordinary Poisson structure, i.e. $P=\frac{1}{2}\,P^{ab}(x)x^*_bx^*_a$ is a bivector field. Then
\begin{equation*}
    \hat P=-\frac{\hbar^2}{2}\,P^{ab}(x)\der{}{dx^b}\der{}{dx^a}
\end{equation*}
differs by the factor of $-\hbar^2$ from $i(P)$. Hence $\D_P=[d,\hat P]=-\hbar^2[d,i(P)]=-\hbar^2 \p_P$ where $\p_P$ is defined by~\eqref{eq.koszulbryl}. Note  that since $d(1)=i(P)(1)=0$, we have $\D_P(1)=0$. Consider the right-hand side of~\eqref{eq.highkoszbydp}. Since here $\hat P$ is a second-order differential operator on the algebra $\O(M)$, the commutator $[\ldots[\D_P,\o_1],\ldots,\o_n]$ gives zero for $n>2$. Consider $n=0, 1, 2$. For $n=0$, we have $\D_P(1)=0$. For $n=1$, we have $[\D_P,\o](1)=\D_P(\o 1)-(-1)^{\tilde \o}\o\D_P(1)=\D_P(\o)=-\hbar^2\p_P(\o)$. Hence
\begin{equation*}
    (-i\hbar)^{-1}[\D_P,\o](1)=-i\hbar\, \p_P(\o)
\end{equation*}
and in the limit $\hbar\to 0$ we get zero. Finally  for $n=2$, we obtain
\begin{multline*}
    (-i\hbar)^{-2}[[\D_P,\o_1],\o_2](1)= [[\p_P,\o_1],\o_2](1)=\\
    \p_P(\o_1\o_2)-\p_P(\o_1)\o_2-(-1)^{\tilde \o_1}\o_1\p_P(\o_2)=[\o_1,\o_2]_P\,,
\end{multline*}
as claimed.
\end{example}

Hence we  conclude that the operator $\D_P$ given by~\eqref{eq.hdelp} indeed generalizes the Koszul--Brylinski operator   $\partial_P$ given by~\eqref{eq.koszulbryl} and formula~\eqref{eq.highkoszbydp} generalizes Koszul's formula~\eqref{eq.binkoszbv}.

\begin{example}\label{ex.difpoiss} {A slight extension of the previous example can be obtained by allowing a linear term in a $\Pinf$-structure: $P=P_1+P_2$, where $P_1=P^a(x)x^*_a$ and $P_2=\frac{1}{2}\,P^{ab}(x)x^*_bx^*_a$. It is convenient to introduce a vector field $Q\in\Vect(M)$ by $Q=Q^a(x)\p_a$, $Q^a=-P^a$. Then $Q$ is homological and the $\Pinf$-structure given by $P$ is a \emph{differential Poisson structure} in the sense that $P_2$ defines a usual (binary) Poisson bracket and the homological field $Q$ is a derivation of the bracket. The operator $\D_P$ corresponding to such $P$ will have the form $\D_P=[d,\hat P]=-i\hbar L_Q-\hbar^2 \p_{P_2}$, where $\p_{P_2}$ is the Koszul-Brylinski operator and $L_Q$ is the (odd) operator of the Lie derivative along $Q$. For the brackets generated by $\D_P$ we get again $\D_P(1)=0$, so the $0$-bracket is zero;  for the $1$-bracket,}
\begin{equation*}
    (-i\hbar)^{-1}[\D_P,\o](1)=L_Q(\o) -i\hbar\, \p_P(\o)\,,
\end{equation*}
{so for $\hbar\to 0$, }
\begin{equation*}
    [\o]_P=L_Q(\o)\,,
\end{equation*}
{and for the $2$-bracket one can see that we obtain}
\begin{equation*}
    [\o_1,\o_2]_P= [\o_1,\o_2]_{P_2}\,, 
\end{equation*}
{the Koszul bracket for the Poisson structure $P_2$\,.}
\end{example}

{A $\Pinf$-structure  $P=P_1+P_2$ of  Example~\ref{ex.difpoiss} is the simplest case of the structure given by the Cattaneo-Felder construction~\cite{cattaneo-felder:relative2007} on the supermanifold $\Pi E^*$, where $E$ is a vector bundle equipped with an ordinary Poisson structure $\bar P$ such that the zero section is coisotropic: if $\bar P$ has only terms of degrees $-1$ and $0$ in fiber coordinates on $E$, then on $\Pi E^*$ this gives a homological vector field of degree $+1$ (making $E^*$ a Lie algebroid) and a compatible Poisson bracket of degree $0$.}

Formula~\eqref{eq.hdelp} with a sketch of a brute-force proof 
were suggested in the appendix to~\cite{tv:highkosz} .
Here we prove Theorem~\ref{thm.DP} in a conceptual way. The underlying idea is that the operator~\eqref{eq.hdelp} should be seen as a ``quantization'' of the   cotangent $\Linf$-algebroid.  We will come back to the main claim in subsection~\ref{subsec.mainstate}. Before that we   need to develop some general theory, which also has  independent interest.

\subsection{Formal $\hbar$-differential operators and their symbols}
\label{subsec.formalop}

Recall
$\hbar$-differential operators.
They can be defined algebraically for an arbitrary commutative (super)algebra over a ring of formal power series in $-i\hbar$ or for  a module over such an algebra. Basically they correspond to a   concept well-known  in the theory of linear partial differential operators (see e.g. Shubin~\cite{shubin:book}). Such are also the examples arising in quantum mechanics.

Here Planck's constant   is treated as a formal parameter and not as some small number.  So the algebraic definition is as follows~\cite{tv:microformal}. An operator is called \emph{$\hbar$-differential} of order $n$ (meaning $\leq n$) if its commutator with the multiplication operator by an element of the algebra is $-i\hbar$ times an $\hbar$-differential operator of order $n-1$, and the operators of negative order are zero. In particular, every $\hbar$-differential operator of order $n$ is a differential operator of   order $n$ in the usual sense. Informally, for a partial differential operator $L$ written in local coordinates, the condition that it is $\hbar$-differential means that every partial derivative occurring in $L$ carries  the factor of $-i\hbar$, and the coefficients can themselves depend on $\hbar$.

In other words, in  a local chart  the algebra of $\hbar$-differential operators is generated by arbitrary  functions $f=f(x)$ and the momentum operators $\hat p_a=-i\hbar \lder{}{x^a}$. There is the ``Heisenberg commutation relation''
\begin{equation}\label{eq.heis}
    [\hat p_a,f]=-i\hbar \der{f}{x^a}\,.
\end{equation}
Here functions are allowed to be formal power series in $\hbar$ (with non-negative powers).
Our operators act  on scalar functions, but   everything generalizes to operators acting on   sections of a vector bundle over    a manifold or supermanifold, and we will use it later.

A \emph{formal $\hbar$-differential operator} is defined as a formal expression (a non-commutative formal power series)
\begin{equation}\label{eq.formalL}
    L=L^0(x)+L^a(x)\hp_a  + L^{a_1a_2}(x)\hp_{a_1}\hp_{a_2} +\ldots
\end{equation}
where the coefficients are smooth functions that are also formal power series in $\hbar$.
As we shall explain, such sums are invariant under changes of variables. We shall also show how they can be seen as actual operators.
But before that, we briefly return to $\hbar$-differential operators, i.e. finite sums.

The commutation relation~\eqref{eq.heis} is homogeneous in $\hp_a$ and $\hbar$ taken together. We introduce a grading  in the algebra of $\hbar$-differential operators   by defining the  \textbf{total degree}  of such an operator as the degree    in $\hp_a$ and $\hbar$. It is a grading,  not a filtration. Elements of  the so obtained graded algebra are $\hbar$-differential operators whose coefficients are polynomials in $\hbar$. Call them \emph{$\hbar$-differential operators of finite type}.  From~\eqref{eq.heis} it is clear that the total degree of an operator of finite type does not depend on its presentation as a non-commutative polynomial in the momentum operators $\hp_a$ with coefficients in the algebra of functions.

\begin{lemma} The total degree of an $\hbar$-differential operator  of finite type does not depend on a choice of coordinates.
\end{lemma}

\begin{proof} \label{lem.transflaw}
The operator $\hp_a$ has the transformation law
\begin{equation*}
    \hp_a=\der{x^{a'}}{x^a}\, \hp_{a'}\,,
\end{equation*}
which is   homogeneous, assuming   the changes of coordinates  not depend on $\hbar$. For a product $\hp_{a_1}\ldots \hp_{a_k}$, we obtain
\begin{equation*}
    \hp_{a_1}\ldots \hp_{a_k}=
    \der{x^{a_1'}}{x^{a_1}}\, \hp_{a'_1}\ldots \der{x^{a_k'}}{x^{a_k}}\, \hp_{a'_k}\,,
\end{equation*}
where it remains to move all the ``new'' momentum operators 
to the right of the coefficients. By induction we see that\footnote{For simplicity, we write
them without signs as if in purely even case, but everything holds in the general super case.}
\begin{equation*}
    \hp_{a_1}\ldots \hp_{a_k}=
    \der{x^{a_1'}}{x^{a_1}} \ldots \der{x^{a_k'}}{x^{a_k}}\,\hp_{a'_1}\ldots \hp_{a'_k}
    + (-i\hbar)R_{k-1}+ (-i\hbar)^2R_{k-2}+\ldots +   (-i\hbar)^kR_0\,,
\end{equation*}
where each $R_s$ is a linear combination of products of exactly $s$ operators $\hp_{a'}$ with some coefficients not depending on $\hbar$ at the left.
The claim immediately follows.
\end{proof}

Now if we have a formal series such as~\eqref{eq.formalL}, a formal power series in both $\hp$ and $\hbar$, we can re-arrange the summation into an infinite sum over the total degree     $n=0,1,2,\ldots$, so that for each $n$ there will be only a finite number of terms:
\begin{multline}\label{eq.formalL2}
    L=\sum_{k=0}^{+\infty} L^{a_1\ldots a_k}(x)\hp_{a_1}\ldots\hp_{a_k}=
    \sum_{k=0}^{+\infty}\sum_{r=0}^{+\infty} (-i\hbar)^rL_r^{a_1\ldots a_k}(x)\,\hp_{a_1}\ldots\hp_{a_k}=\\
    \sum_{n=0}^{+\infty}\Bigl(L_0^{a_1\ldots a_n}(x)\hp_{a_1}\ldots\hp_{a_n}+
    (-i\hbar)L_1^{a_1\ldots a_{n-1}}(x)\hp_{a_1}\ldots\hp_{a_{n-1}}+
    \ldots +(-i\hbar)^nL_n^0(x)\Bigr) \\
    \equiv  \sum_{n=0}^{+\infty} L^{[n]}\,.
\end{multline}
Here $L^{[n]}$ is  a finite-type $\hbar$-differential operator of total degree $n$, the component of total degree $n$ of an operator $L$. Hence formal $\hbar$-differential operators make an algebra, which is the formal completion of the graded algebra of   $\hbar$-differential operators of finite type.

\begin{lemma}
\begin{enumerate}
  \item A formal $\hbar$-differential operator  $L$ gives rise to a formal power series
\begin{equation}\label{eq.Lmodh}
    L\pmod \hbar  = \sum_{n=0}^{+\infty} L_0^{a_1\ldots a_n}(x)\,p_{a_1}\ldots\,p_{a_n}\,,
\end{equation}
which is a well-defined formal function on $T^*M$,    $p_a$ is identified with $\hp_a \pmod \hbar$;
  \item There is a well-defined  action of formal $\hbar$-differential operators on  functions which are formal power series in $\hbar$ and this action   respects  grading;
  \item There is a well-defined  action of formal $\hbar$-differential operators on  functions of the form $e^{\ih \la g(x)}$, where $g(x)$ is a formal power series in $\hbar$,   which  gives   products   of $e^{\ih \la g(x)}$ with   formal power series in both $\hbar$ and $\la$.
\end{enumerate}
\end{lemma}

\begin{proof}
In the proof of Lemma~\ref{lem.transflaw} we observed the transformation law of a typical summand of a formal $\hbar$-differential operator. Modulo $\hbar$, it is the same as for the corresponding monomial in the variables $p_a$. This proves part 1. For part 2, if $L=\sum_{n=0}^{\infty} L^{[n]}$ as in~\eqref{eq.formalL2} and $f=\sum_{n=0}^{\infty}(-i\hbar)^nf_n$ is the expansion of a function $f\in\fun(M)[[\hbar]]$, then
\begin{equation*}
    L(f)=\sum_{n=0}^{\infty}\sum_{r+s=n} (-i\hbar)^sL^{[r]}(f_s)
\end{equation*}
and it remains to observe that
\begin{multline*}
    L^{[r]}(f_s)=L_0^{a_1\ldots a_r}\,\hp_{a_1}\ldots\hp_{a_r}(f)+
    (-i\hbar)L_1^{a_1\ldots a_{r-1}}\,\hp_{a_1}\ldots\hp_{a_{r-1}}(f)+
    \ldots +(-i\hbar)^rL_r^0\,f=\\
    (-i\hbar)^r\Bigl(L_0^{a_1\ldots a_r}\,\p_{a_1}\ldots\p_{a_r}f+
      L_1^{a_1\ldots a_{r-1}}\,\p_{a_1}\ldots\p_{a_{r-1}}f+
    \ldots + L_r^0\,f\Bigr)\,,
\end{multline*}
so it is of degree $r$ in $\hbar$. The action of an operator $L^{[r]}$ of total degree $r$ on a function of degree $s$ in $\hbar$ gives a function of degree $r+s$ in $\hbar$. As for part 3, consider the action of $\hp_a$ on a function of the form $f e^{\ih \la g}$\,; we obtain
\begin{equation*}
    \hp_a\bigl(fe^{\ih \la g}\bigr)=(-i\hbar\p_a f  +\la f \p_ag)\,e^{\ih \la g}\,.
\end{equation*}
Similarly,
\begin{equation*}
    \hp_a\hp_b\bigl(fe^{\ih \la g}\bigr)=
    \bigl((-i\hbar)^2\p_a\p_b f  + (-i\hbar)\la(\p_a f\p_b g + \p_b f\p_a g+ f\p_a\p_b g)
    +\la^2f\p_a g\,\p_b g\bigr)\,e^{\ih \la g}\,.
\end{equation*}
By induction we can see that always $L^{[n]}\bigl(fe^{\ih \la g}\bigr)=P_{\hbar,\la}\,e^{\ih \la g}$, where $P_{\hbar,\la}$ is a homogeneous polynomial of total degree $n$ in $\hbar$ and $\la$ whose term of degree $r$ in $\hbar$ is  polynomial of degree $n-r$ in   partial derivatives of $g$ of order $\leq r+1$ and linear in partial derivatives of $f$.
\end{proof}

Instead of writing an object modulo $\hbar$ we  also write $\lim\limits_{\hbar\to 0}$.


The formal function on $T^*M$ defined by formula~\eqref{eq.Lmodh} is called the \emph{principal symbol} of a formal $\hbar$-differential operator $L$ and will be denoted $\symb(L)$,
\begin{equation}\label{eq.symbL}
    \symb(L):= \sum_{n=0}^{+\infty} L_0^{a_1\ldots a_n}(x)\,p_{a_1}\ldots\,p_{a_n}\,,
\end{equation}
if $L$ is given by~\eqref{eq.formalL}, \eqref{eq.formalL2}.  It is different from a (coordinate-dependent) \emph{full symbol} of $L$  obtained by formally replacing $\hp_a$ by $p_a$ in the expansion~\eqref{eq.formalL} without setting $\hbar$ to zero. Also this principal symbol is different from the principal symbol of a differential operator of order $\leq n$ (which is a homogeneous polynomial of degree $n$ corresponding to the top order  derivatives).

\begin{lemma} \label{lem.symbcommut}
For formal $\hbar$-differential operators,
\begin{equation}\label{eq.symbprod}
    \symb(AB)=\symb(A)\symb(B)\,
\end{equation}
(hence $\symb([A,B])=0$    for all $A$, $B$).
The commutator $[A,B]$ is always divisible by $\hbar$ and
\begin{equation}\label{eq.symbcommut}
    \symb(i\hbar^{-1}[A,B])=\{\symb(A),\symb(B)\}\,,
\end{equation}
where at the right-hand side there is the Poisson  bracket on $T^*M$.
\end{lemma}
\begin{proof}
Formula~\eqref{eq.symbprod} is obvious from the definition of principal symbol and the rules of multiplication of formal power series. Since the product of functions on $T^*M$ is commutative, $\symb([A,B])=[\symb(A),\symb(B)]=0$. To prove formula~\eqref{eq.symbcommut},  since both  commutator of operators and   Poisson bracket on $T^*M$ satisfy the Leibniz identity, it is sufficient to check it on the generators such as $\hp_a$ and $f(x)$, for which it becomes obvious.
\end{proof}

The  following definitions were introduced by Th.~Voronov~\cite{tv:higherder} as a modification of the construction of Koszul~\cite{koszul:crochet85}. {(See Remark~\ref{rem.hist} below.)}
\begin{definition}
For an operator $L$ on an algebra,
\begin{align}
\{f_1,\ldots,f_n\}_{L,\hbar}&:=
 (-i\hbar)^{-n}  \left[\ldots [L,f_1],\ldots,f_n\right](1)
\label{eq.qubrack}
\intertext{is the   \emph{quantum $n$-bracket} and }
   \{f_1,\ldots,f_n\}_{L} &:=
   (-i\hbar)^{-n}  \left[\ldots [L,f_1],\ldots,f_n\right](1)\pmod \hbar \label{eq.clbrack}
\end{align}
is the  \emph{classical $n$-bracket} generated by $L$. Here $n=0,1,2,3, \ldots$
\end{definition}

Here 
$f_i$ are functions on a supermanifold or elements of an abstract commutative superalgebra.
One has to assume that an $n$-fold commutator $\left[\ldots [L,f_1],\ldots,f_n\right]$ in the above formulas is divisible by $(-i\hbar)^{n}$.
In particular, this makes sense for formal $\hbar$-differential operators on supermanifolds as defined here.

\begin{example}[$0$-, $1$- and $2$-brackets]
The    quantum  $0$-bracket  is simply
\begin{align}\label{eq.brack0}
    \{\void\}_{L,\hbar}&=L(1)\,;
\intertext{for the  quantum $1$-bracket   take $[L,f](1)=L(f1)-(-1)^{\Lt\ft}fL(1)=L(f)-L(1)f$, hence}
\label{eq.brack1}
    \{f\}_{L,\hbar}&=i\hbar^{-1}\bigl(L(f)-L(1)f\bigr)\,;
\intertext{similarly, for the $2$-bracket one has}
\label{eq.brack2}
    \{f,g\}_{L,\hbar}&=-\hbar^{-2} \bigl(L(fg)-L(f)g-(-1)^{\Lt\ft}fL(g)+ L(1)fg\bigr)\,.
\end{align}
\end{example}

Quantum brackets are themselves (formal)  differential operators in each argument {(but not $\hbar$-differential)}, moreover it is known~\cite{tv:higherder} that for any $n$ the $n$-bracket generates the $(n+1)$-bracket as a {``quantum correction''} to the Leibniz rule:
\begin{multline}\label{eq.quantleib}
    \{f_1,\ldots,f_{n-1},fg\}_{L,\hbar}=\{f_1,\ldots,f_{n-1},f\}_{L,\hbar}\,g+
    (-1)^{\e}f\,\{f_1,\ldots,f_{n-1},g\}_{L,\hbar}\\
    + (-i\hbar)\{f_1,\ldots,f_{n-1},f,g\}_{L,\hbar}\,,
\end{multline}
where $(-1)^{\e}=(-1)^{(\tilde L +\ft_1+\ldots +\ft_{n-1})\ft}$.
Modulo $\hbar$ the extra term  disappears and the resulting classical brackets become multiderivations. Hence they must correspond to a Hamiltonian.

\begin{theorem} \label{thm.symbol}
Let $L$ be a formal $\hbar$-differential operator. The Hamiltonian $H$ for the classical brackets generated by $L$ is  the principal symbol of $L$, $H=\symb(L)$.
\end{theorem}
\begin{proof} We need to prove the identity
\begin{equation}\label{eq.clbrackham}
    (-i\hbar)^{-n}  \left[\ldots [L,f_1],\ldots,f_n\right](1)\pmod \hbar =\{\ldots \{H,f_1\},\ldots,f_n\}_{|M}
\end{equation}
where $H=\symb(L)\in \fun (T^*M)$, $f_i\in\fun(M)$, and the brackets at the right-hand side are the Poisson brackets on $T^*M$. Indeed we observe that for an arbitrary formal $\hbar$-differential operator $A$,
\begin{equation*}
    A(1)\pmod \hbar = \symb(A)_{|M}
\end{equation*}
(application to $1$ gives the free term of the operator $A$, which modulo $\hbar$ is the zeroth term in the expansion of the principal symbol). Hence, by induction,
\begin{multline*}
    (-i\hbar)^{-n}  \left[\ldots [L,f_1],\ldots,f_n\right](1)\pmod \hbar= \\
    \symb\bigl(i\hbar^{-1} \bigl[\,i\hbar^{-1} [\ldots i\hbar^{-1}[L,f_1],\ldots,f_{n-1}]\,,f_n\,\bigr]\bigr)_{|M}=\\
    \bigl\{\,\symb\bigl(i\hbar^{-1}  \bigl[\ldots i\hbar^{-1}[L,f_1],\ldots,f_{n-1}\bigr]\bigr)\,, f_n\,\bigr\}_{|M}=\\
    \ldots = \bigl\{\bigl\{\ldots \bigl\{\symb(L),f_1\bigr\},\ldots,f_{n-1}\bigr\},f_n\bigr\}_{|M}
\end{multline*}
where we used formula~\eqref{eq.symbcommut}. 
\end{proof}

{Let a formal $\hbar$-differential operator $\D$  be   odd and satisfy $\D^2=0$. Since $\D^2=\frac{1}{2}[\D,\D]$,  from Theorem~\ref{thm.symbol} and Lemma~\ref{lem.symbcommut} it follows that the corresponding odd Hamiltonian $H:=\symb(\D)$ satisfies $\{H,H\}=0$ and therefore the classical brackets~\eqref{eq.clbrack}, with $L=\D$, form an $\Sinf$-algebra. Moreover, $\D^2=0$ implies (in fact, is equivalent to) that the quantum brackets~\eqref{eq.qubrack} generated by $\D$  satisfy themselves the higher Jacobi identities and so form an $\Linf$-algebra~\cite{tv:higherder}, though they are no longer multiderivations of the associative product and satisfy  instead  relation~\eqref{eq.quantleib}. Such a structure   introduced in~\cite{tv:higherder} is called in~\cite[\S5]{tv:microformal} an  \emph{$S_{\infty,\hbar}$-algebra}.}

\begin{remark}[on history and terminology]\label{rem.hist}
{A sequence of multilinear operations $\Phi^n_L$  for an operator $L$  on a graded-commutative  algebra %
was first introduced   by Koszul~\cite{koszul:crochet85}. They are basically~\eqref{eq.qubrack} without division by $(-i\hbar)^n$. Koszul himself was mostly interested in the case of an odd second-order operator $\D$. He established an identity linking the failure of Jacobi for $\Phi^2_{\D}$   with $\Phi^3_{\D}$ and $\Phi^3_{\D^2}$ and basically meaning that if $\D^2=0$, then the operation $\Phi^2_{\D}$ satisfies Jacobi identity up to a chain homotopy with $\Phi^3_{\D}$ as the  homotopy operator. Hence if   $\D$ is of second order and $\D^2=0$, it generates an odd Poisson (=Schouten or Gerstenhaber) bracket. Such operators later became known as  Batalin-Vilkovisky (BV) operators. A graded commutative algebra   with a BV operator is called a \emph{BV algebra}. } 

{
In this paper, we use the name ``Batalin-Vilkovisky operator'' in a broader sense  including operators of higher order.
The study of an  analogue  of BV algebras based on  a higher-order operator $\D$ such that $\D^2=0$  was initiated by O.~Kravchenko in a seminal paper~\cite{kravchenko:bv}. She called the obtained structure a \emph{$BV_{\infty}$-algebra}. Kravchenko noticed that the condition $\D^2=0$ is equivalent to the   sequence of higher Jacobi identities for the sequence of   brackets $\Phi^n_{\D}$~\cite[Prop.~2]{kravchenko:bv}, so that $BV_{\infty}$ implies $\Linf$. 
(Note that there are also more general notions of a homotopy Gerstenhaber~\cite{avoronov:hga} and homotopy BV algebras~\cite{tamarkintsygan:bv}, which we do not need here.) In~\cite{tv:higherder}, Th.\,Voronov put forward a general algebraic mechanism leading to $\Linf$-algebras, for which Koszul's construction of brackets and Kravchenko's theorem are a particular example. The modification of Koszul's definition by the factor of $(-i\hbar)^n$ was suggested in~\cite{tv:higherder} to obtain a deformation of an $\Sinf$-structure.  We took the notion of $S_{\infty,\hbar}$-algebras from~\cite{tv:microformal}. It is   very close to $BV_{\infty}$-algebras in the sense of Kravchenko, but the difference is that it is based on $\hbar$-differential operators, which is essential for   our purposes.}
\end{remark}

\subsection{Main statement. ``Quantum'' and ``classical'' higher Koszul brackets.}\label{subsec.mainstate}

We can apply the above considerations to the situation where on a supermanifold $M$ there is an $\Sinf$-structure specified by an odd Hamiltonian $H$ satisfying the ``classical master equation'' $\{H,H\}=0$.
If there is an odd formal $\hbar$-differential operator satisfying $\D^2=0$ such that the odd brackets on $M$ coincide with the classical brackets generated by $\D$, the operator $\D$ is called a \emph{Batalin--Vilkovisky operator} for a given $\Sinf$-structure. {(See also the remark above.)}
By Theorem~\ref{thm.symbol}, then $\symb(\D)=H$. Hence finding $\D$ for a given $\Sinf$-structure{, i.e. lifting it to an $S_{\infty,\hbar}$-structure,}  is a ``quantization problem''  and $\D$ is not unique (since an operator $\D$ contains more data than its principal symbol).

Now return to our {particular} problem.
Our goal is to find a Batalin--Vilkovisky operator for the higher Koszul brackets on $\O(M)=\fun(\Pi TM)$ induced by a $\Pinf$-structure on $M$.

Recall from subsection~\ref{subsec.highkosz} that the odd  master Hamiltonian  for the higher Koszul brackets is
\begin{equation}\label{eq.masterkosz2}
    H_P =\{D, P^*\}\,,
\end{equation}
where $D=dx^a p_a$ and $P^*\in \fun(T^*(\Pi TM))$ is obtained from $P\in \fun(\Pi T^*M)\subset \fun(T^*(\Pi T^*M))$ by the Mackenzie--Xu transformation. See formulas~\eqref{eq.hamlichn},\eqref{eq.masterkosz}. In local coordinates, if $P=P(x,x^*)$, then $P^*=P(x,\pi)$, with $\pi_a$ being the momenta canonically conjugate with $dx^a$.

Both Hamiltonians $D$ and $P^*$ have natural  quantizations.

\begin{example}
The operator $-i\hbar\, d=-i\hbar\, dx^a\der{}{x^a}$ is a quantization of the Hamiltonian $D=dx^a p_a$, i.e. $\symb(-i\hbar\, d)=D$.
\end{example}

\begin{example}
For an arbitrary function $P\in \fun (\Pi T^*M)$, the operator $\hat P=P(x,-i\hbar \der{}{dx})$ is a quantization of the Hamiltonian $P^*=P(x,\pi)$, i.e. $\symb(\hat P)=P^*$.
\end{example}

The following statement is a generalization of Cartan's identity.

\begin{lemma}\label{lem.verscartan}
For arbitrary $T,S\in \fun (\Pi T^*M)$,
\begin{equation}\label{eq.verscartan}
    \bigl[[d,\hat T], \hat S\bigr]=\lsch T,S\rsch^{\widehat{\ }}\,.
\end{equation}
\end{lemma}
\begin{proof}
Direct calculation using formula~\eqref{eq.Phatint}. Or, alternatively, choose a volume element $\rho$ and take fiberwise $\hbar$-Fourier transform of the left-hand side of~\eqref{eq.verscartan} and obtain $[[\d,T],S]$ where $\d=\d_{\rho}$ is the divergence operator on multivector fields. But since $\d$ is a differential operator of second order, $[[\d,T],S]$  is a differential operator of order zero, i.e.   a multivector field. Then $[[\d,T],S]=[[\d,T],S](1)=\d(TS)-\d(T)S-(-1)^{\tilde T}T\d(S)=\lsch T,S\rsch$.
\end{proof}

Now everything is ready for the main statement.

\begin{theorem}[a stronger version of Theorem~\ref{thm.DP}]
The operator $\D_P=[d,\hat P]$ is a Batalin--Vilkovisky operator for the $\Sinf$-structure on $\Pi TM$ induced by a $\Pinf$-structure on $M$ (i.e. for higher Koszul brackets).
\end{theorem}
\begin{proof} There are two statements: that $\D_P$ indeed generates the Koszul brackets (as the classical brackets) and that $\D_P^2=0$. For the first statement, consider the principal symbol of $\D_P$. By Lemma~\ref{lem.symbcommut},
\begin{equation*}
    \symb(\D_P)=\symb([d,\hat P])=\symb((-i\hbar)^{-1}[-i\hbar\, d, \hat P)=\{\symb(-i\hbar\, d),\symb(\hat P)\}=\{D,P^*\}\,.
\end{equation*}
(Note that $d$ is not an $\hbar$-differential operator, so one cannot mistakenly decide that $\symb([d,\hat P])$ is zero!)
For the second statement, consider $\D_P^2$. We have
\begin{equation*}
     [[d,\hat P],\hat P]=0
\end{equation*}
(since by Lemma~\ref{lem.verscartan}, the left-hand side is $\lsch P,P\rsch^{\widehat{}}$). By applying the commutator with $d$, we obtain
\begin{equation*}
    0=[d[d,\hat P],\hat P]]=\pm [[d,\hat P],[d,\hat P]]
\end{equation*}
(since $d^2=0$). But this is exactly $[\D_P,\D_P]\equiv 2 \D_P^2=0$\,.
\end{proof}

Note that we also obtain  \emph{quantum Koszul brackets}  as the quantum brackets generated by $\D_P$,
\begin{equation}\label{eq.quantkosz}
    [\o_1,\ldots,\o_n]_{P,\hbar}:=(-i\hbar)^{-n}[\ldots[\D_P,\o_1],\ldots,\o_n](1)\,.
\end{equation}
This is a useful notion even for the ordinary Poisson case.

\begin{example}\label{ex.ordpoiss2}
Let $P$ be a bivector field defining an  ordinary  Poisson structure on a supermanifold $M$.
From calculations in Example~\ref{ex.ordpoiss},  $\D_P=-\hbar^2\p_P$ (where $\p_P$ is the Koszul--Brylinski operator) and we can see that
\begin{align*}
    [\void]_{P,\hbar}&=0\,,\\
    [\o]_{P,\hbar}&= -i\hbar\,\p_P(\o)\,,\\
    [\o_1,\o_2]_{P,\hbar}&=[\o_1,\o_2]_{P}\,.
\end{align*}
Quantum and classical Koszul $2$-brackets coincide because $\D_P$ is of second order. All the higher brackets are zero. In particular, from the quantum viewpoint the Koszul-Brylinski operator (with the factor of $-i\hbar$) is itself part of the sequence of brackets and the   known derivation property~\cite{koszul:crochet85}
\begin{equation}\label{eq.binkoszderdp}
    \p_P[\o_1,\o_2]_P = -[\p_P\o_1,\o_2]_P+(-1)^{\tilde\o_1+1}[\o_1,\p_P\o_2]_P\,
\end{equation}
becomes part of   higher Jacobi identities for   quantum Koszul brackets.
\end{example}

\subsection{``Quantum cotangent $\Linf$-bialgebroid''.}
\label{subsec.quantcotang}

Construction of the BV operator $\D_P$ generating the higher Koszul brackets on $\O(M)=\fun(\Pi TM)$ can be interpreted as a ``quantization'' of the cotangent $\Linf$-algebroid structure. We will see how to extend that to a ``quantization'' of  an  $\Linf$-bialgebroid.

The cotangent $\Linf$-algebroid is indeed an $\Linf$-bialgebroid. In the manifestation on $\Pi TM$, that means that $D$, the master Hamiltonian for the de Rham differential, and $H_P$, the master Hamiltonian for the higher Koszul brackets, make a commuting pair. This is equivalent to the odd Hamiltonian
\begin{equation}\label{eq.cldoub}
    D_{t}=D+tH_P
\end{equation}
depending on parameter $t\in \RR$ satisfying $\{D_t,D_t\}=0$ for all $t$. In terms of the brackets, this is the derivation property for $d$ and all the higher Koszul brackets.

This   lifts to the ``quantum level'' as follows.

\begin{theorem}
For every $t$, the formal $\hbar$-differential operator
\begin{equation}\label{eq.bvbialg}
    \hat D_t= -i\hbar\, d + t\D_P\,.
\end{equation}
is a Batalin--Vilkovisky operator  which is a quantum lift of   the  Hamiltonian $D_{t}=D+tH_P$.
The operator $\hat D_t$ can be also written as
\begin{equation}\label{eq.bvbialg2}
    \hat D_t=  e^{-\ih t\hat P}(-i\hbar\,d)e^{\ih t\hat P}\,.
\end{equation}
\end{theorem}
\begin{proof}
It is clear that   $\symb(\D_t)=D_t$\,, so  $\hat D_t$ is a quantization of the master Hamiltonian $D_t$. We need to show that $\hat D_t^2=0$. Indeed, $d^2=0$ and we know that $\D_P^2=0$, so we need $[d,\D_P]=0$. But $[d,\D_P]=[d,[d,\hat P]]=(\ad d)^2(\hat P)=0$.
Finally, we need to establish the identity
\begin{equation*}
   -i\hbar\, d + t\D_P=e^{-\ih t\hat P}(-i\hbar\,d)e^{\ih t\hat P}\,.
\end{equation*}
Indeed,
\begin{multline*}
   e^{-\ih t\hat P}(-i\hbar\,d)e^{\ih t\hat P}= e^{-\ih t\ad \hat P}(-i\hbar\,d)=
    -i\hbar\,d -\ih t(\ad \hat P)(-i\hbar\,d)= -i\hbar\,d-t\ad(\hat P)(d)=\\
    -i\hbar\,d-t[\hat P, d]= -i\hbar\,d+t[d,\hat P]=
    -i\hbar\, d + t\D_P
\end{multline*}
because
\begin{equation*}
    (\ad \hat P)^2(d)=[\hat P,[\hat P, d]]=\pm \lsch P,P\rsch ^{\widehat{}}=0\,.
\end{equation*}
\end{proof}

We considered the cotangent $\Linf$-bialgebroid in the manifestation on $\Pi TM$ and for it constructed a quantization. Let us see how this can be done in the dual picture on $\Pi T^*M$ and how these pictures will be related ``on the quantum level''.

On $\Pi T^*M$, the roles of brackets and homological vector field is swapped compared to $\Pi TM$: instead of $d$, there is the canonical Schouten bracket; and instead of the higher Koszul brackets, there is the Lichnerowicz differential $d_P$.  A Batalin--Vilkovisky operator for the Schouten bracket is $-\hbar^2\d$, where $\d=\d_{\rho}$ is the divergence operator  constructed with the help of some volume element $\rho$ on $M$. Hence the operator
\begin{equation} \label{eq.wrongbv}
    -\hbar^2\d +t(-i\hbar)d_P
\end{equation}
seems a natural choice for a Batalin--Vilkovisky operator on $\Pi T^*M$ for the $\Linf$-bialgebroid structure. However, \eqref{eq.wrongbv} does not work because this operator does not in general square to zero. Indeed, although $\d^2=0$ and $d_P^2=0$, we have for $[\d,d_P]$
\begin{equation*}
    [\d,d_P](T)=\d\lsch P,T\rsch + \lsch P, \d (T)\rsch=
    -\lsch \d (P), T\rsch -\lsch P, \d(T)\rsch +\lsch P, \d (T)\rsch =
    -\lsch \d (P), T\rsch\,.
\end{equation*}
Hence unless $\d (P)$ is zero, the operators $\d$ and $d_P$ do not commute and
$(-\hbar^2\d +t(-i\hbar)d_P)^2\neq 0$. One recognizes in $\d (P)$ a representative of the  \emph{modular class}  of a $\Pinf$-structure:  a cohomology class   $[\d(P)]\in H^*(\Mult(M),d_P)$ defined with the help of a volume element $\rho$ but independent of a choice of $\rho$. So a different choice of  $\rho$ does not solve the problem if the modular class $[\d(P)]$ is nonzero.
\begin{remark}
The modular class of a $\Pinf$-structure $[\d(P)]\in H^*(\Mult(M),d_P)$ is directly analogous to the constructions for ordinary Poisson manifolds~\cite{weinstein:modular}, Lie algebroids~\cite{evens-lu-weinstein:1999}  and $Q$-manifolds~\cite{lyakhovich:mosman1},\cite{tv:qman-esi}.
\end{remark}

The situation can be remedied by taking on $\Pi T^*M$ the   operator
\begin{equation}\label{eq.bigdp}
    \hat D_P:=-i\hbar(d_P+\d(P))
\end{equation}
instead of $-i\hbar d_P$. This does not change the principal symbol. The term $-i\hbar \d(P)$ is a ``quantum correction''.

\begin{lemma} One can express
\begin{equation}\label{eq.bigdp2}
    \hat D_P=-i\hbar[\d, P]\,.
\end{equation}
The operator $\hat D_P$ has square zero.
\end{lemma}
\begin{proof}
We have $[\d, P](T)=\d(PT)-P\d(T)=\d(P)T+P\d(T)+\lsch P,T\rsch -P\d(T)=\d(P)T +\lsch P,T\rsch=(d_P+\d(P))(T)$. Now, $(d_P+\d(P))^2=d_P^2+(\d(P))^2+[d_P,\d(P)]=[d_P,\d(P)]=d_P(\d(P))=\lsch P,\d(P)\rsch=\pm \d\lsch P,P\rsch=0$.
\end{proof}

\begin{theorem}
For every $t$, the   operator
\begin{equation}\label{eq.bvbialgdual}
    \hat D^*_t= - \hbar^2\, \d + t\hat D_P
\end{equation}
is a Batalin--Vilkovisky operator for the $\Linf$-bialgebroid structure on  $\Pi T^*M$.
The operator $\hat D^*_t$ can be also written as
\begin{equation}\label{eq.bvbialg2dual}
    \hat D^*_t=  e^{-\ih t  P}(-\hbar^2\,\d)e^{\ih t  P}\,.
\end{equation}
\end{theorem}
\begin{proof}
The principal symbol of $\hat D^*_t$ is
\begin{equation*}
    \symb(\hat D^*_t)=\symb( - \hbar^2\, \d) - t\symb(i\hbar  d_P+i\hbar \d(P)) =
    \symb( - \hbar^2\, \d) - t\symb(i\hbar  d_P)=
    D^*_t H_P^*\,.
\end{equation*}
as claimed. For $(\hat D^*_t)^2$ we have
\begin{equation*}
    (\hat D^*_t)^2= (- \hbar^2\, \d + t\hat D_P)^2=-\hbar^2\,t\,[\d,\hat D_P] =-\hbar^2\,t\,[\d,-i\hbar[\d, P]]=0\,.
\end{equation*}
Finally,
\begin{equation*}
    e^{-\ih t  P}(-\hbar^2\,\d)e^{\ih t  P}=e^{-\ih t \ad P}(-\hbar^2\,\d)=-\hbar^2\,\d
    -\ih t \ad P (-\hbar^2\,\d)=\hat D^*_t\,.
\end{equation*}
\end{proof}

Classical description of the cotangent $\Linf$-bialgebroid structure is symmetric with respect to $\Pi TM$ and $\Pi T^*M$, and one can be obtained from another by the Mackenzie--Xu transformation.

On the level of the constructed BV operators,
our construction loses this symmetry. To amend
this, we will introduce operators acting on half-densities instead of functions.



\section{``Symmetric theory''}\label{sec.further}

In this section we first develop a general theory
of operators acting on densities for dual vector bundles. Here the main results are Theorem~\ref{thm.dualdo} and Theorem~\ref{thm.dualquapull}. Secondly, we explain how to get brackets on functions from an operator acting on a one-dimensional module (such as half-densities). Finally, we arrive to  Theorem~\ref{thm:BVsym}, which gives the desired fully symmetric (in both manifestations on $\Pi TM$ and $\Pi T^*M$) description of the quantum cotangent $\Linf$-bialgebroid.

\subsection{Dualization for operators on vector bundles}\label{subsec.dualiz}

Our goal here is to develop convenient tools for working with operators on dual vector   bundles  such as $\Pi TM$ and $\Pi T^*M$. It will be based on fiberwise $\hbar$-Fourier transform that we will introduce below.

Let $E\to M$ be a vector bundle, $E$ and $M$ are (super)manifolds.
Informally, we want to describe the dual space $(\fun(E))'$ to the space of functions $\fun(E)$ in terms of geometric objects on the dual bundle $E^*\to M$. It is convenient to work in a slightly more general setting, namely to consider  densities instead of functions.

On a vector bundle $E$ consider   densities of weight $(\la,\mu)$ as objects of the form $$f(x,u)Dx^{\,\la}Du^{\,\mu}$$ in local coordinates, where $x^a$, $u^i$ are coordinates on the base and the fiber respectively. Here by $Dx$ we denote the (Berezin) coordinate volume element which transforms according to the formula $Dx=(Dx/Dx')Dx'$, where $Dx/Dx'=\Ber (\partial   x / \partial x')$, and similarly for $Du$.
~\footnote{On   supermanifolds there are more types of orientation conditions because    different combinations of signs in
   $\Ber_{\a,\b} J:= (\sign\det J_{00})^{\a}\, (\sign\det J_{11})^{\b}\, \Ber J$\,,
where $J$ is the Jacobi matrix of a change of coordinates,   give different analogues of  $\det J$ and $|\det J|$ of the ordinary case. Change of variables in Berezin integral includes $\Ber_{1,0}J$, hence $\det J_{00}>0$ is the orientability condition required for integration of densities of the form $f(x)Dx$ over a supermanifold. Integration of pseudodifferential forms $\o(x,dx)$ requires a different condition, $\Ber J>0$. See~\cite{tv:git}.  There are the corresponding types of densities whose transformation laws include powers of $\Ber_{\a,\b} J$. In the case of  super  fiber bundles, the number of orientation and density types  becomes even larger. These distinctions are not relevant for our purposes, so we  completely ignore them and in particular  will write $Dx^{\la}$ etc. instead of a more refined notation.
}
Denote this space of densities $\Dens_{\la,\mu}(E)$. In particular, $\Dens_{\la,\la}(E)=\Dens_{\la}(E)$, where $\Dens_{\la}(E)$ is the usual space of $\la$-densities on a supermanifold $E$. (Unlike~\cite{tv:genstuli} we will be considering operators on densities of fixed weight, not on the algebra of densities.)

Let $f(x,u)Dx^{\,\la}Du^{\,\mu}\in \Dens_{\la,\mu}(E)$. Introduce its \emph{fiberwise $\hbar$-Fourier transform} by:
\begin{equation}\label{eq.fourier}
    F[f(x,u)Dx^{\,\la}Du^{\,\mu}] = \left(\int_{E_x}e^{-\ih u^iw_i} f(x,u)\,Du\right)Dx^{\la}Dw^{1-\mu}\,.
\end{equation}
Here $w_i$ are coordinates in the fiber of $E^*$ such that the bilinear form $u^iw_i$ is invariant. In~\eqref{eq.fourier}, we use the identification $Du^{\mu-1}=Dw^{1-\mu}$.
It establishes an isomorphism
\begin{equation}\label{eq.fourier2}
    F\co \Dens_{\la,\mu}(E) \to \Dens_{\la,1-\mu}(E^*)\,,
\end{equation}
which holds  for all $\la$ and $\mu$.
Together with the natural identification $(\Dens_{\la,\mu}(E))'\cong \Dens_{1-\la,1-\mu}(E)$ (here the dual of a functional space  can mean its ``smooth part'' or we can simply agree provisionally not to keep track of the smoothness), this gives rise to an isomorphism, which we denote by the same letter $F$,
\begin{equation}\label{eq.fourier3}
    F\co  (\Dens_{\la,\mu}(E))' \to \Dens_{1-\la,\mu}(E^*)\,
\end{equation}
(here we applied~\eqref{eq.fourier2} for $1-\la,1-\mu$ instead of $\la,\mu$).
To put it differently, the isomorphism~\eqref{eq.fourier3} is equivalent to a non-degenerate pairing
\begin{equation}\label{eq.fourierpair}
    \Dens_{\la,\mu}(E)\times \Dens_{1-\la,\mu}(E^*)\to \CC
\end{equation}
given by the integral
\begin{equation}\label{eq.fourierpairint}
    \int_{E}e^{-\ih u^iw_i} f(x,u)\,g(x,w)\,Du\,Dw\,Dx\,.
\end{equation}
Here $f(x,u)Dx^{\la}Du^{\mu}\in \Dens_{\la,\mu}(E)$ and $g(x,w)Dx^{1-\la}Dw^{\mu}\in \Dens_{1-\la,\mu}(E^*)$\,.

The  fiberwise $\hbar$-Fourier transform~\eqref{eq.fourier} or the pairing that it provides~\eqref{eq.fourierpairint} make it possible to consider the formal {duals} to linear maps between functions or densities on vector bundles  as linear maps between the suitable densities on the dual bundles: if
\begin{equation}\label{eq.lineop}
    L\co \Dens_{\la,\mu}(E_1)\to \Dens_{\la,\mu}(E_2)\,,
\end{equation}
then
\begin{equation}\label{eq.lineopdual}
    L^*\co \Dens_{1-\la,\mu}(E_2^*)\to \Dens_{1-\la,\mu}(E_1^*)\,.
\end{equation}
If $L'$ stands for the usual formal dual treated as a map
\begin{equation}
    \Dens_{1-\la,1-\mu}(E_2)\cong (\Dens_{\la,\mu}(E_2))'\to (\Dens_{\la,\mu}(E_1))'\cong \Dens_{1-\la,1-\mu}(E_1)\,,
\end{equation}
then
\begin{equation}
    L^*=F_1\circ L'\circ F_2^{-1}\,,
\end{equation}
where $F_1=F_{E_1}$ and $F_2=F_{E_2}$ are the Fourier transforms for $E_1$ and $E_2$ respectively.

One particular useful case if that of half-densities, where an $L\co \Dens_{\frac{1}{2}}(E_1)\to \Dens_{\frac{1}{2}}(E_2)$ induces the $L^*\co \Dens_{\frac{1}{2}}(E_2^*)\to \Dens_{\frac{1}{2}}(E_1^*)$\,. Another option is to make use of a volume element for the base $M$ in order to identify ``base $\la$-densities'' with ``base  $0$-densities''. This would allow to introduce a $\rho$-dependent dual $L^*_{\rho}$, where $\rho\in \Vol(M)$ is a chosen volume element. If for example
\begin{equation}\label{eq.lineopfun}
    L\co \fun(E_1)\to \fun(E_2)\,,
\end{equation}
 then
\begin{equation}\label{eq.lineopfundual}
    L^*_{\rho}=\rho^{-1}\circ L^*\circ\rho\co \fun(E_2^*)\to \fun(E_1^*)\,.
\end{equation}

We will apply these constructions to $\hbar$-differential operators on vector bundles and also to some special Fourier integral type operators (see below).

Before doing that, consider gradings. Any object on a vector bundle has a natural $\ZZ$-grading, which following~\cite{tv:graded} we call \emph{weight} and denote $\w$. So on $E$, with fiber coordinates $u^i$, we have $\w(u^i)=+1$ and $\w(\lder{}{u^i})=-1$. Also, if $w_i$ are the corresponding dual fiber coordinates for $E^*$, then $\w(w_i)=-1$. (This is counting weights relative $E$, $\w=\w_{E}$. Of course, for $E^*$ considered on its own, $\w_E=-\w_{E^*}$.)

Dealing with densities, one needs to take into account the weights of coordinate volume elements. If we have $n$ even and $m$ odd variables among $u^i$, then
\begin{equation*}
    \w(Du)=n-m\,.
\end{equation*}
Respectively, if we need to use $\d$-functions, then
\begin{equation*}
    \w(\d(u))=-n+m\,.
\end{equation*}
The integral symbol $\int$ has weight zero.
The  fiberwise $\hbar$-Fourier transform~\eqref{eq.fourier} and  the pairing~\eqref{eq.fourierpair} behave nicely with respect to grading.
\begin{lemma}\label{lemma.fouriergrad}
The  fiberwise $\hbar$-Fourier transform
\begin{equation}\label{eq.fou}
    F\co \Dens_{\la,\mu}(E) \to \Dens_{\la,1-\mu}(E^*)
\end{equation}
preserves   grading given by $\w_E$. Formula~\eqref{eq.fourierpair} gives a non-degenerate pairing of elements of weight $\a$ in $\Dens_{\la,\mu}(E)$ with elements of weight $-\a$ in $\Dens_{1-\la,\mu}(E^*)$.
\end{lemma}
\begin{proof}
Note that weights of densities need not be integral: for elements of $\Dens_{\la,\mu}(E)$ they take  values in $\ZZ+\mu(n-m)$, where $\dim E_x=n|m$. So $\w(f(x,u)Dx^{\la}Du^{\mu})=\#u^i+\mu(n-m)$. One immediately checks that  $\w(F[f(x,u)Dx^{\la}Du^{\mu}])=\#u^i+n-m+(1-\mu)(-n+m)=\#u^i+\mu(n-m)=\w(f(x,u)Dx^{\la}Du^{\mu})$.               (Note that $\w(u^iw_i)=0$.) Similarly for the pairing given by~\eqref{eq.fourierpair}: if $\#u^i(f)=k$, then for the integral to be non-zero, it should be that $\#w_i(g)=k$, i.e. $\w_E(g)=-k$. So the subspace of elements of weight $k+\mu(n-m)$ in  $\Dens_{\la,\mu}(E)$ is non-degenerately paired with the subspace of elements of weight $-k+\mu(-n+m)$ in $\Dens_{1-\la,\mu}(E^*)$.
\end{proof}

Denote by $\DO(\Dens_{\la,\mu}(E))$ the algebra of finite type $\hbar$-differential operators on $(\la,\mu)$-densities on   $E$ with fiberwise-polynomial coefficients.

There are two natural gradings defined on elements of this algebra: one by total degree of operator (see subsection~\ref{subsec.formalop}) and another by weight coming from the vector bundle structure of $E$. 
We can write them as~\footnote{We attach the subscript $E$ to stress the relation with the bundle $E$.}
\begin{align}\label{eq.totdege}
\deg_E(L)&=\# \hp_a +\# \hp_i +\# \hbar\\
    \w_E(L)&=\# u^i-\# \hp_i \label{eq.we}
\intertext{(where $\#$ denotes the degree in given variables).  It is possible to introduce another invariant grading as the sum $\deg^*_E:=\deg_E+\w_E$\,, or}
\deg^*_E(L)&=\# \hp_a +\#u^i +\# \hbar\,.
\end{align}
We can consider the algebra $\DO(\Dens_{\la,\mu}(E))$ as bi-graded by $(\deg_E, \deg^*_E)$. Note that $\deg_E, \deg^*_E\geq 0$.
Define now  the algebra $\fDO(\Dens_{\la,\mu}(E))$ as the  formal completion  of the bi-graded algebra $\DO(\Dens_{\la,\mu}(E))$. Elements of  $\fDO(\Dens_{\la,\mu}(E))$ are formal $\hbar$-differential operators with fiberwise formal coefficients. We simply call them ``formal $\hbar$-differential operators'' on $E$.

\begin{theorem}\label{thm.dualdo}
Taking dual $L\mapsto L^*$  is an anti-isomorphism of   algebras
\begin{equation}\label{eq.dualop}
    \fDO(\Dens_{\la,\mu}(E))\to \fDO(\Dens_{1-\la,\mu}(E^*))\,,
\end{equation}
\begin{equation}\label{eq.antiisom}
    (L_1L_2)^*=(-1)^{\tilde L_1\tilde L_2}\,L_2^*L_1^*\,,
\end{equation}
which maps the subalgebra $\DO(\Dens_{\la,\mu}(E))$ on the subalgebra $\DO(\Dens_{1-\la,\mu}(E^*))$ and swaps the bi-grading: for every $L\in \DO(\Dens_{\la,\mu}(E))$
\begin{equation}\label{eq.bigradlstar}
    \deg_{E^*}(L^*)=\deg^*_E(L) \quad \text{and} \quad \deg^*_{E^*}(L^*)=\deg_E(L)\,.
\end{equation}
In the limit $\hbar\to 0$, the map $L\mapsto L^*$ induces an algebra isomorphism
\begin{equation}\label{eq.algmx}
    \fun(T^*E)\to \fun(T^*(E^*))
\end{equation}
which is the pull-back by the Mackenzie--Xu antisymplectomorphism $T^*(E^*) \to  T^*E$\,.
\end{theorem}
\begin{proof} The fact that it is an anti-isomorphism follows from the definition, as $(L_1L_2)'=(-1)^{\tilde L_1\tilde L_2}L_2'L_1'$. To prove the rest, one needs to see the images of the local generators.

By integration by parts, one gets $(\lder{}{x^a})^*=-\lder{}{x^a}$, $(\lder{}{u^i})^*=\ih w_i$, $(u^i)^*=(-1)^{\itt}(\hi)\lder{}{w_i}$. In other words,
\begin{equation}\label{eq.qumx}
   (\hp_a)^*=-\hp_a \,, \quad (\hp_i)^*= w_i \,, \quad (u^i)^*= \hp^i\,(-1)^{\itt}
\end{equation}
and since $(f(x))^*=f(x)$, it is an anti-isomorphism of $\fun(M)$-algebras. The statements concerning gradings follow immediately.

The relation with the Mackenzie--Xu transformation $T^*(E^*) \to  T^*E$ is clear from the above formulas (as, without hats, they give the coordinate expression of this map).
Indeed, since $L\mapsto L^*$ is an anti-isomorphism of $\fun(M)$-algebras, it remains so modulo $\hbar$. Hence the induced map~\eqref{eq.algmx} is  an isomorphism of commutative $\fun(M)$-algebras, so it must arise from some diffeomorphism  $T^*(E^*) \to  T^*E$ over $M$. That it coincides with the Mackenzie--Xu diffeomorphism, follows from~\eqref{eq.qumx}. Finally, since the map of algebras of formal $\hbar$-differential operators is an anti-isomorphism, it takes the commutator of operators to the negative of the commutator. Modulo $\hbar\to 0$ this gives that the Poisson bracket on $T^*E$ is mapped to the negative of the Poisson  bracket  on $T^*(E^*)$. (So we in particular recover the antisymplectomorphism property of the Mackenzie--Xu map, which is of course clear directly.)
\end{proof}

\begin{remark} In the classical limit, the bi-grading by $(\deg_E, \deg^*_E)$ becomes the bi-grading  associated with the double vector bundle structure:
\begin{equation*}
    \begin{CD} T^*E\cong T^*(E^*) @>>> E^* \\
    @VVV @VVV \\
    E @>>> M\,,
    \end{CD}
\end{equation*}
$\deg_E=\#p_a+\#p_i=\#p_a+\#w_i$\,, and $\deg_{E^*}=\#p_a+\#p^i=\#p_a+\#u^i$\,.
\end{remark}

As  mentioned, the ``most symmetric'' picture is obtained for half-densities, $\la=\mu=1/2$.
An example particularly interesting for us in this paper is that of $E=\Pi TM$. Then $D(x,dx)$ is an invariant volume element, so on $\Pi TM$ one can identify densities of any weight with just functions, i.e. pseudodifferential forms on $M$. On the other hand, for $\Pi T^*M$, the volume element $D(x,x^*)$ transforms as $(Dx)^2$. Hence half-densities on $\Pi T^*M$ can be identified with with multivector densities on $M$, i.e. (pseudo)integral forms on $M$. They can be written as $\s(x,x^*)Dx$. The pairing~\eqref{eq.fourierpairint} takes the form
\begin{equation}\label{eq.pairforms}
    \int_{\Pi TM}D(x,dx)\,e^{-\ih dx^a x^*_a}\, \o(x,dx)\,\s(x,x^*)\,.
\end{equation}
Here $\o(x,dx) \in \O(M)\cong\Dens_{\frac{1}{2}}(\Pi TM)$ and $\s(x,x^*)Dx\in \Mult(M,\Vol M)\cong \Dens_{\frac{1}{2}}(\Pi T^*M)$\,.

\begin{example}
We have for operators on the algebra $\O(M)$, from the pairing~\eqref{eq.pairforms}, $(\lder{}{x^a})^*=-\lder{}{x^a}$, $(dx^a)^*=-\hi\lder{}{x^*_a}(-1)^{\at}$ and $(\lder{}{dx^a})^*=\ih x^*_a$. We obtain, in particular,
\begin{equation}\label{eq.dstar}
    (-i\hbar\,d)^*=-\hbar^2\d
\end{equation}
where
\begin{equation}\label{eq.deltaonhalf}
    \d=(-1)^{\at}\der{}{x^a}\der{}{x^*_a}
\end{equation}
and
\begin{equation}\label{eq.pstar}
    P\left(x,\hi\der{}{dx}\right)^*=P(x,x^*)\,.
\end{equation}
(Note that here unlike section~\ref{sec.operator}, $\d$ is an operator on multivector densities and does not require a choice of a volume form on $M$ for its definition.)
\end{example}

Consider now an operator $L\co \Dens_{\la,\mu}(E_2)\to \Dens_{\la,\mu}(E_1)$, $L\co f_2(x,u_2)Dx^{\la}Du_2^{\mu}\mapsto f_1(x,u_1)Dx^{\la}Du_1^{\mu}$,   expressed by an integral formula
\begin{equation}\label{eq.qupull}
    f_1(x,u_1)=\int Du_2\Dbar w_2\, e^{\ih(S(u_1|w_2)-u_2w_2)} f_2(x,u_2)\,.
\end{equation}
Here we use $u_1^i$ and $u_2^{\a}$ for fiber coordinates for $E_1$ and $E_2$, $w_1=(w_{1i})$ and $w_2=(w_{2\a})$ for fiber coordinates in $E_1^*$  and $E_2^*$. The function $S(u_1|w_2)$ in the phase of the exponential is called ``generating function''. It  is  a formal power series in $w_2$. (It has a non-trivial transformation law under changes of coordinates that guarantees the invariance of the integral transform.) Such integral operators are a particular case   of the   operators introduced in~\cite{tv:microformal} and interpreted  there as certain ``quantum pullbacks''. In particular, they include ordinary pullbacks. {See also~\cite{operovermap}.}
\begin{example}\label{ex.linpull}
Suppose $\Phi\co E_1\to E_2$ is a morphism of vector bundles over $M$, $u_2^{\a}=u_1^i\Phi_i^{\a}(x)$. Let $S(u_1|w_2):=u_1^i\Phi_i^{\a}(x)w_{2\a}$. Then it is easy to see that the integral operator~\eqref{eq.qupull} with such a generating function $S$ is the pullback of functions, $L=\Phi^*\co \fun(E_2)\to \fun(E_1)$.
\end{example}


\begin{theorem}\label{thm.dualquapull}
If $L\co \Dens_{\la,\mu}(E_2)\to \Dens_{\la,\mu}(E_1)$ is an operator of the form~\eqref{eq.qupull}, then its dual  $L^*\co \Dens_{1-\la,\mu}(E_1^*)\to \Dens_{1-\la,\mu}(E_2^*)$ is an operator of the same form   given by the integral formula
\begin{equation}\label{eq.qupulldual}
    g_2(x,w_2)=\int Dw_1\Dbar u_1\, e^{\ih(S^*(w_2|u_1)-u_1w_1)} g_1(x,w_1)\,,
\end{equation}
where $S^*(w_2|u_1)=S(u_1|w_2)$.
\end{theorem}

\begin{proof} Directly: we need to prove  the equality
\begin{equation*}
    \int Dx Du_1Dw_1 \,e^{-\ih\,u_1w_1}f_1(x,u_1)g_1(x,w_1)=
    \int Dx  Du_2Dw_2 \,e^{-\ih u_2w_2}f_2(x,u_2)g_2(x,w_2)
\end{equation*}
for arbitrary $f_2(x,u_2)$ and $g_1(x,w_1)$,
where $f_1(x,u_1)$ and $g_2(x,w_2)$ are given by   formulas~\eqref{eq.qupull} and~\eqref{eq.qupulldual}. By substituting, we arrive at the identity.
\end{proof}

Theorem~\ref{thm.dualquapull} is a ``quantum analogue'' of Theorem 8 in~\cite{tv:microformal}.

\begin{example} Continuing in the setup of Example~\ref{ex.linpull}, we see that the dual to the pullback by a vector bundle morphism $\Phi$ over $M$ is the pullback by the dual   morphism $\Phi^*$. Hence the presented construction makes it possible to extend the notion of the dual operator to more general situations that naturally occur.  For example, to the case of a non-linear fiberwise map between vector bundles, a situation typical for $\Linf$-algebroids.
\end{example}

\begin{remark}
The fact that Fourier transform for odd variables acts like ``Hodge star operator'' was observed in early years of supergeometry. See Berezin~\cite{berezin:difforms};  Voronov--Zorich~\cite{tv:ivb},  also~\cite{tv:git}.
Fourier transform of geometric objects on vector bundles was considered in~\cite{tv:ivb}, \cite{tv:class}. In~\cite{tv:ivb}, a non-standard $\ZZ$-grading of pseudodifferential forms on a vector bundle was introduced (from modern viewpoint this extra grading comes from the double vector bundle structure of $\Pi TE$) and it was shown that fiberwise Fourier transform of forms preserves that grading. This is analogous to our Lemma~\ref{lemma.fouriergrad}. That the divergence operator $\d$ on multivector densities is  ``dual'' to the de Rham differential $d$ is a classical fact in differential geometry. It reappeared in supergeometry in Bernstein--Leites construction of integral forms~\cite{berl:int},\cite{berl:pdf} and in connection with Batalin--Vilkovisky formalism~\cite{hov:semi},\cite{tv:formsandsymplectic}.
Novel in this subsection are Theorems~\ref{thm.dualdo} and \ref{thm.dualquapull}, as well as  bi-grading   $(\deg_E,\deg_{E^*})$ of formal $\hbar$-differential operators on a  vector bundle $E$.
\end{remark}

\begin{remark}
Since we treat $\hbar$ as a formal parameter, one may ask in which sense oscillating exponentials in $\hbar$-Fourier transform, as well as in the integral operators~\eqref{eq.qupull}, are understood. These exponentials with $\hbar$ in the denominator can be defined as formal symbols satisfying natural properties, as explained in~\cite{tv:microformal}. An axiomatic theory of formal oscillatory integrals was developed by A.~Karabegov~\cite{karabegov:formal}.
\end{remark}

\subsection{More on BV operators and brackets}

In section~\ref{sec.operator} we considered brackets generated by an operator acting on a commutative algebra, with the standard example of the algebra of functions on a supermanifold. Now we need operators
defined on a (locally)  one-dimensional free module, an example of which is the case of operators acting on densities of a fixed weight. Particularly interesting for us is the case of half-densities.

Let $S$ be such a module over an algebra $A$. Let $\s$ be some basis element of $S$.
Consider an arbitrary operator $L\co S\to S$. In concrete examples it will be a formal $\hbar$-differential operator. We can define an operator $L_{\s}\co A\to A$ depending on a choice of $\s$, by
\begin{equation}\label{eq.lsigma}
    L_{\s}(f):=\s^{-1}L(\s f)\,.
\end{equation}
Consider quantum and classical brackets on $A$ generated by  $L_{\s}$ by \eqref{eq.qubrack} and \eqref{eq.clbrack}. Denote them $\{f_1,\ldots,f_n\}_{L,\s,\hbar}$ and $\{f_1,\ldots,f_n\}_{L,\s}$.

\begin{lemma} For any $n$,
\begin{align}
\{f_1,\ldots,f_n\}_{L,\s,\hbar}& =
 (-i\hbar)^{-n} \s^{-1} \left[\ldots [L,f_1],\ldots,f_n\right](\s) \\
   \{f_1,\ldots,f_n\}_{L,\s} &=
   (-i\hbar)^{-n}  \s^{-1}\left[\ldots [L,f_1],\ldots,f_n\right](\s)\pmod \hbar
\end{align}
\end{lemma}
\begin{proof}
Observe that $[L_{\s},f]=[L,f]_{\s}$ and then apply induction.
\end{proof}

A natural question is about the dependence of the constructed brackets on a choice of $\s$. Suppose $\s'=e^{g}\s$. We can apply a method from~\cite{tv:microformal} to compare the homological vector fields on the algebra $A$ for the brackets corresponding to $\s$ and $\s'$. For brackets ``without dash'',
\begin{equation*}
    f\mapsto f+ \e e^{-\ih f}L_{\s}(e^{\ih f})
\end{equation*}
and for brackets with dash,
\begin{equation*}
    f\mapsto f+ \e e^{-\ih f}L_{\s'}(e^{\ih f})=
    f+ \e e^{-g}e^{-\ih f}L_{\s}(e^ge^{\ih f})=
    f+ \e  e^{-\ih( f+\hi g)}L_{\s}(e^{\ih (f+\hi g)})\,.
\end{equation*}
In other words, the ``new'' homological field is obtained by the shift of argument by a fixed function $\hi g$. This is an $\Linf$-isomorphism which is the identity modulo $\hbar$. In particular,
classical brackets do not depend on a choice of $\s$.

\subsection{Application to cotangent $\Linf$-bialgebroid. Discussion}\label{subsec.applcot}
We come back to the construction of a Batalin--Vilkovisky operator for the cotangent $\Linf$-algebroid (for a given $\Pinf$-structure). As we know, it is actually an $\Linf$-bialgebroid. The difference with subsection~\ref{subsec.quantcotang} is that
we use operators acting of half-densities (instead of functions).

\begin{theorem} \label{thm:BVsym}
The Hamiltonians defining the cotangent $\Linf$-bialgebroid structure lift to  mutually dual Batalin--Vilkovisky operators on $\Pi T^*M$ and $\Pi TM$:
\begin{equation}\label{eq.bv1}
    e^{-\ih P}\circ (-\hbar^2\,\d)\circ e^{\ih P}
\end{equation}
(acting on  multivector densities on $M$)
and
\begin{equation}\label{eq.bv2}
    e^{-\ih \hat P}\circ (-i\hbar\, d)\circ e^{\ih \hat P}
\end{equation}
(acting on forms). Here $P=P(x,x^*)$ and $\hat P=P(x,-i\hbar \der{}{dx})$ are as before.
\end{theorem}

{Half-densities on $\Pi T^*M$ are multivector densities on $M$ and half-densities on $\Pi TM$ because of  canonical volume element  can be identified with functions, i.e. forms on $M$.}

We do not discuss the most general and ``abstract'' notion  of $\Linf$-bialgebroid or it quantum version. See different approaches and analysis in~D.~Bashkirov and A.~Voronov~\cite{bashkirov-voronov:bv}, \cite{bashkirov-voronov:lbalg} and Th.~Voronov~\cite{tv:microformal}, \cite{tv:linfbialg}. One thing that we would like to note, is the role of the modular class (which was appearing in section~\ref{sec.operator} {when we used an asymmetric picture}). In the description based on half-densities, {we expect it to arise}  as an obstruction to a ``quantum lift'' of the anchor. If such as obstruction vanishes, there are dual descriptions of the quantum anchor on $\Pi T^*M$ and $\Pi TM$ given by mutually dual integral operators as in subsection~\ref{subsec.dualiz}.

Let me elaborate this point. It is convenient to speak about general $\Linf$-algebroids. If an $\Linf$-algebroid structure is given in a vector bundle $E\to M$, its manifestation on $\Pi E$ is just a homological vector field $Q\in \Vect(\Pi E)$.  Higher anchors  as   multilinear maps from   $E$ to  $TM$\,---\,part of an $\Linf$-algebroid structure\,---\,combine into a single non-linear fiberwise $Q$-map $\Pi E\to \Pi TM$. See e.g.~\cite[Lem.~1]{tv:microformal}. In the dual manifestation on $\Pi E^*$, an $\Linf$-algebroid structure in $E$ becomes a collection of higher Lie-Schouten brackets, i.e. an $\Sinf$-structure on $\Pi E^*$. The ``dual anchor'' then is a thick $\Sinf$-morphism $\Pi T^*M \tto \Pi E^*$ in Voronov's sense~\cite{tv:nonlinearpullback} inducing an $\Linf$-morphism $\fun(\Pi E^*) \to \fun(\Pi T^*M)$~\cite{tv:nonlinearpullback},\cite{tv:microformal},\cite{tv:highkosz}. (The particular case of $E=T^*M$ and  an $\Linf$-morphism between the higher Koszul brackets and the canonical Schouten bracket was a problem posed by Khudaverdian-Voronov in~\cite{tv:higherpoisson} and solved in the above-cited works.)  We see the quantum version   as follows. A \emph{quantum $\Linf$-algebroid} structure in $E$ is given by an odd $\hbar$-differential operator $\hat H$ on $\Dens_{1/2}(\Pi E)$, of total degree $\deg \hat H=+1$ (see~\ref{subsec.dualiz},  where $\hat H^2=0$. Its classical limit, i.e. the principal symbol, gives a usual $\Linf$-algebroid  structure in $E$. A \emph{quantum anchor} 
has to be defined as an integral operator of  type~\eqref{eq.qupull}, $\O(M)\to \Dens_{1/2}(\Pi E)$ intertwining $-i\hbar d$ and $\hat H$.  In the dual description, a dual quantum anchor will be an integral operator $\Dens_{1/2}(\Pi E^*)\to \Mult(M)$ intertwining $\hat H^*$ and $-\hbar^2 \delta$. This is a lifting of the standard classical picture. We expect the modular class of $E$ play a role for such a lifting.  We hope to explore  this  and  the   ``bi-'' case  elsewhere.



\end{document}